\documentclass{amsart}

\usepackage[
	theoremdefs,
	final
]{allergy}

\usepackage{enumitem}

\usepackage{booktabs}
\usepackage{pdflscape}
\usepackage{afterpage}
\usepackage{capt-of}
\usepackage{multirow}
\usepackage{colortbl}
\usepackage{booktabs}
\usepackage{subcaption}
\usepackage{tocvsec2}

\usepackage{pgfplots}
\pgfplotsset{compat=newest}
\usepackage{standalone}

\usepackage[numbers]{natbib} 

\usepackage{amsmath,amssymb}

\providecommand{\norm}[1]{\lVert#1\rVert}
\providecommand{\abs}[1]{\lvert#1\rvert}
\providecommand{\vect}[1]{\boldsymbol{#1}}

\newcommand{\ud}{\mathrm{d}}
\newcommand{\udd}{\,\ud}

\newcommand{\RR}{{\mathbb R}}

\newcommand{\ZZ}{{\mathbb Z}}

\newcommand*\GL{\mathrm{GL}}

\newcommand*\SO{\mathrm{SO}}
\providecommand{\so}{\mathfrak{so}}

\newcommand*\g{\mathfrak{g}}

\newcommand*\Group{\mathsf{G}}
\newcommand*\one{\mathrm{1}}

\newcommand{\DLA}[1]{DLA${}^{#1}$}

\newcommand{\inv}{^{-1}}

\newcommand*\Lag{\mathcal{L}}

\newcommand*\Torus{\RR/\ZZ}

\newcommand*\thera{\tau}
\newcommand*\matA{\mathsf{A}}
\newcommand*\matAbar{\overline{\matA}}
\newcommand*\flow{\Phi}
\newcommand*\perPhi{\flow(1)}

\newcommand{\vz}{\dot{\zz}}

\newcommand{\vx}{\dot{\xx}}
\newcommand{\xx}{\boldsymbol{x}}
\newcommand{\qq}{\boldsymbol{q}}
\newcommand{\pp}{\boldsymbol{p}}

\newcommand*\vv{v}

\newcommand{\Fx}{\boldsymbol{F}}

\newcommand*\zz{\xi}
\newcommand*\kker{\boldsymbol{k}}
\newcommand*\FF{\boldsymbol{f}}

\newcommand*\mass{}

\newcommand*\mul{\,}

\newcommand*\trans{^{\mathsf{T}}}
\newcommand*\consmat{{A}}

\newcommand*\yy{\boldsymbol{y}}
\newcommand*\hstiff{\kappa}
\newcommand*\kstiff{\boldsymbol{\kker}_{0}}

\newcommand*\hh{h}

\newcommand{\smallV}{{V}}

\newcommand*\blambda{\boldsymbol{\lambda}}
\newcommand*\dq{\dot{\qq}}

\newcommand*\Upstairs{\mathcal{M}}
\newcommand*\Downstairs{\mathcal{N}}
 
\title{What makes nonholonomic integrators work?}

\author{Klas Modin}
\address{Department of Mathematical Sciences, Chalmers University of Technology and University of Gothenburg, SE-412 96 Gothenburg, Sweden}
\email{klas.modin@chalmers.se}

\author{Olivier Verdier}
\address{Department of Computing, Mathematics and Physics, Western Norway University of Applied Sciences, Bergen, Norway}
\email{olivier.verdier@hvl.no}

\begin{document}

\maketitle

\begin{abstract}
A nonholonomic system is a mechanical system with velocity constraints not originating from position constraints; 
rolling without slipping is the typical example.
A nonholonomic integrator is a numerical method specifically designed for nonholonomic systems.
It has been observed numerically that many nonholonomic integrators exhibit excellent long-time behaviour when applied to various test problems.
The excellent performance is often attributed to some underlying discrete version of the Lagrange--d'Alembert principle.
Instead, in this paper, we give evidence that \emph{reversibility} is behind the observed behaviour.
Indeed, we show that many standard nonholonomic test problems have the structure of being foliated over reversible integrable systems.
As most nonholonomic integrators preserve the foliation and the reversible structure, near conservation of the first integrals is a consequence of reversible KAM theory.
Therefore, to fully evaluate nonholonomic integrators one has to consider also \emph{non}-reversible nonholonomic systems.
To this end we construct perturbed test problems that are integrable but no longer reversible (with respect to the standard reversibility map).
Applying various nonholonomic integrators from the literature to these problems we observe that no method performs well on all problems.
This further indicates that reversibility is the main mechanism behind near conservation of first integrals for nonholonomic integrators.
A list of relevant open problems is given.


\end{abstract}

\maxtocdepth{section}
\tableofcontents

\section{Introduction}

Many models in mechanics and physics are described by dynamical systems with constraints.
If the constraints do not originate from position constraints (so called \emph{holonomic} constraints), the system is called \emph{nonholonomic}.
A typical example is a disc rolling on a surface without slipping.
The governing differential equations are obtained from the \emph{Lagrange--d'Alembert principle} (see \autoref{sec:nonholcoupled}), which is \emph{not} a variational principle (contrary to Hamilton's principle in Lagrangian mechanics).
Therefore, nonholonomic systems do not, in general, preserve a symplectic structure, although total energy is conserved.\footnote{
	An open question is whether energy conservation is the \emph{only} geometric property of nonholonomic systems.
	That is, given a system on a manifold $\mathcal{N}$ with a first integral, is it always possible to view it as the underlying ODE of a nonholonomic system?~\cite{commML}
}\
Motivated by the success of symplectic integrators for Hamiltonian systems, various discrete versions of the Lagrange--d'Alembert principle have nevertheless been suggested, with the objective of deriving ``structure preserving'' time-stepping algorithms for nonholonomic systems~\cite{CoMa2001,Co2002,McPe2006,FeIgMa2008,FeIgMa2009,KoMaSu2010,KoMaDiFe2010}.
Such algorithms are often called \emph{nonholonomic integrators}; they are observed to nearly conserve first integrals when applied to some standard nonholonomic test problems.
The cause of the near conservation is often attributed to the discrete Lagrange--d'Alembert principle.

In this paper we give numerical and theoretical results suggesting that \emph{reversibility} lies behind the observed good behaviour of nonholonomic integrators, regardless of any underlying discrete Lagrange--d'Alembert principle.
Our results in fact reveal that the standard nonholonomic test problems have a bias: they are all reversible integrable.
We therefore construct a family of nonholonomic test problems that are still integrable but not reversible (that is, not reversible with respect to a standard reversibility map).
We apply several nonholonomic integrators from the literature on the new test problems.
None of them exhibit structure preservation for all problems.

The underlying philosophy of nonholonomic integrators is well summarised by \citet[\S\!~1]{CoMa2001} as follows: ``by respecting the geometric structure of nonholonomic systems, one can create integrators capturing the essential features of this kind of systems.''
Because of our limited geometric understanding of nonholonomic dynamics, it is, however, unclear whether nonholonomic integrators at all possess special properties making them superior to other methods (in the same sense as symplectic integrators for Hamiltonian systems possess properties making them superior to non-symplectic methods).
Except for exact conservation of momentum maps corresponding to horizontal symmetries~\cite[\S\!~5]{CoMa2001}, there are no theoretical results pertaining to structure preserving properties of nonholonomic integrators 
(by contrast, the excellent long-time behaviour of symplectic integrators for Hamiltonian systems and reversible integrators for reversible system is fully explained by KAM theory in combination with backward error analysis, see~\citet*{HaLuWa2006} and references therein).
Nonholonomic integrators nevertheless often display ``very good energy behaviour''~\cite[\S\!~1]{FeJiMa2015}.
Such statements are based on experimental evidence---how well the integrators perform on standard test problems.
Thus, a nonholonomic integrator is ``structure preserving'' if it nearly conserves the first integrals of such test problems over long integration times.
With this definition, if we extend the test problem suite to include our unbiased nonholonomic problems, then, as far as we know, there are no structure preserving nonholonomic integrators (all tested integrators fail to be structure preserving for some of the problems). 
Our aspiration for the new set of unbiased test problems is to serve the community as a more complete suite for evaluating structure preservation of nonholonomic integrators.




We now summarize the contributions in the paper:

\begin{itemize}

\item We show that five classical nonholonomic test problems are part of a larger family of \emph{nonholonomically coupled systems} (\autoref{sec:nonholcoupled}, \autoref{sec:nonholexamples}).

\item We show that a subset of nonholonomically coupled systems (that includes the classical test problems) are foliations over reversible integrable systems (\autoref{prop:fibrevinteg}, \autoref{sec:reversibility}, \autoref{sec:fibrations_over_int_sys}).
We thereby obtain a new family of reversible integrable nonholonomic systems that extends existing systems.

\item We use the result in \autoref{prop:fibrevinteg} together with reversible KAM theorem to explain the excellent long-time behaviour of nonholonomic integrators observed experimentally (\autoref{sec:KAM}).

\item We propose new, perturbed test problems for nonholonomic integrators (\autoref{sec:nonholexamples}).
While still integrable, these problems are not reversible and therefore avoid experimental bias.

\item We carry out numerical experiments with five commonly used nonholonomic integrators on both reversible (biased) and non-reversible (unbiased) test problems (\autoref{sec:numerical_examples}).
The behaviour in the numerical simulations is consistent with our predictions from the theory (\autoref{sub:explanation}).
(One of the methods still yields good long-time behaviour on one of the unbiased problems without explanation, but this method fails for other test problems.)

\end{itemize}

The paper is organised as follows.
In \autoref{sec:nonholcoupled} we describe a family of nonholonomic systems; this family contains some of the classical nonholonomic systems in the literature.
In \autoref{sec:numerical_examples} we run numerical experiments on some particular systems in that family, and measure the conservation of some integrals of those systems.
In \autoref{sec:KAM} we show how reversible KAM theory in combination with \autoref{prop:fibrevinteg} explain near conservation of integrals.
In \autoref{sub:explanation} we then verify our theory against the numerical simulations.
In \autoref{sec:reversibility} and \autoref{sec:fibrations_over_int_sys} we prove results leading to \autoref{prop:fibrevinteg}.

\medskip
\noindent\textbf{Acknowledgements.}
This project has received funding from the European Union’s Horizon 2020 research and innovation programme MSC-RISE project \href{https://www.ntnu.edu/imf/chips}{CHiPS}, from the Swedish Foundation for Strategic Research under grant agreement ICA12-0052, from the Swedish Foundation for International Cooperation in Research and Higher Eduction (STINT) grant No PT2014-5823, and from the Swedish Research Council (VR) grant No 2017-05040.
The authors acknowledge the hospitality of the \href{http://ifs.massey.ac.nz/}{Institute for Fundamental Sciences} of Massey University, and of the mathematics departments of Chalmers and NTNU, where some of this research was conducted.

\section{Nonholonomically coupled systems}
\label{sec:nonholcoupled}

The Lagrange--d'Alembert principle states that a motion curve $\qq(t)$, for a system with Lagrangian function $\Lag(\qq,\dq)$ and nonholonomic constraints $A\paren{\qq}\dq = 0$ fulfils
\begin{equation}\label{eq:LA}
	\delta \int_a^{b} \Lag\paren[\big]{\qq(t),\dot\qq(t)} \udd t = 0,
\end{equation}
for all virtual displacements $\delta\qq$ with $A\paren[\big]{\qq(t)}\delta\qq(t) = 0$.
Throughout this paper we assume that the Lagrangian is of the form
\begin{equation}
	\Lag(\qq,\dot\qq) = \frac{1}{2}\dot\qq^\top \dot\qq - V(\qq),
\end{equation}
for a potential function $V$.
The equations of motion are then given by
\begin{equation}\label{eq:eq_of_motion}
	\begin{split}
		\ddot\qq &= - \nabla V(\qq) + A(\qq)^\top \boldsymbol{\lambda} \\
		0 &= A(\qq)\dot\qq
	\end{split}
\end{equation}
where $\blambda \in \RR^{r}$ are the Lagrange multipliers (see~\cite{Co2002,Bl2015} for details).


As we shall see in \autoref{sec:nonholexamples}, the continuously varying transmission, the nonholonomic oscillator and nonholonomic particle, the knife edge, vertical rolling disk,
are all part of a greater family of nonholonomically coupled systems,
where two independent subsystems are coupled through the constraints.

\begin{definition}
\label{def:nonholcoupledsys}
A \emph{nonholonomically coupled system} is a nonholonomic system with Lagrangian of the form
\begin{align}
	\Lag(\underbrace{\xx,\zz}_{\qq},\underbrace{\vx,\vz}_{\dot\qq}) = \frac{1}{2} \vx\trans \mass \vx + \frac{1}{2} \vz^2 - U(\xx)- \smallV(\zz)
	, \qquad \xx\in\RR^{n-1}, \;\zz\in\RR
	,
\end{align}
and constraints of the form
\begin{align}\label{eq:con_coupled}
	\consmat(\zz) \mul \vx = 0
	,
\end{align}
where, for any $\zz$, the matrix $\consmat(\zz)$ has a kernel of dimension one.
The $(\zz,\vz)$ subsystem is called the \emph{driving system}.
\end{definition}

\begin{remark}
  Note that, in the examples of \autoref{sec:nonholexamples}, some components of $\xx$, or $\zz$, may be periodic (see \autoref{tab:summary}).
  In the rest of the paper we will nevertheless assume that $\xx \in \RR^{n-1}$ and $\zz \in \RR$, for the sake of simplicity.
\end{remark}

\begin{remark}
  Note that the matrix $A$ in equation~\eqref{eq:con_coupled} is not the same as the one appearing in~\eqref{eq:eq_of_motion}.
  The difference is that now $A$ depends only on $\zz$, and applies only on $\vx$.
  This slight abuse of notation should not be confusing.
\end{remark}

Notice that the driving system is a self-contained unconstrained Lagrangian system.
As indicated, one may think of it as the ``driver'' for the remaining system.

We write the total energy $H$ as
\begin{align}\label{eq:coupled_tot_energy}
H(\xx,\zz,\vx,\vz) = \frac{1}{2} \vx\trans \mass \vx + U(\xx) + \hh(\zz,\vz)
,
\end{align}
where $\hh(\zz, \vz)\coloneqq \frac{1}{2}\vz^2 + \smallV(\zz)$ is the energy of the driving system.

Given a nonholonomically coupled system, let $\kker(\zz)$ be a kernel vector of $\consmat(\zz)$ such that $\norm{\kker(\zz)} = 1$.
We then define
\begin{align}
	\vv \coloneqq \vx\trans \kker(\zz)
	.
\end{align}
Since $\kker(\zz)$ spans the kernel of $\consmat(\zz)$, it follows from the constraint equation~\eqref{eq:con_coupled} that~$\vx = \vv \kker(\zz)$.
Also note that since \(
\frac{1}{2}\vx\trans \mass \vx = \frac{1}{2}  \vv^2
\),
we have
\begin{equation}
  H(\xx,\zz,\vx,\vz) = \frac{1}{2}\vv^2 + U(\xx) + \hh(\zz,\vz)
  .
\end{equation}

Both the total energy $H$ and the energy $\hh$ of the driving system are first integrals, so we obtain that
\begin{align}
	\frac{\ud}{\ud t}H &= \frac{\ud}{\ud t}\left( \frac{1}{2} \vv^2 + U(\xx) \right) + \underbrace{\frac{\ud}{\ud t} \hh}_{0} \\
	&= \vv \frac{\ud\vv}{\ud t} - \Fx(\xx)\trans \dot\xx = 0
	 ,
\end{align}
where $\Fx(\xx) \coloneqq -\nabla U(\xx)$.
We now use $\dot\xx = \vv \kker(\zz)$, which, if $\vv\neq 0$, gives
\begin{align}
\label{eq:diffeqv}
	\dot\vv = {\kker(\zz)}\trans \Fx(\xx)
	.
\end{align}
Note that, even though this derivation assumes $v=0$, one can check that \eqref{eq:diffeqv} is still valid when $v=0$ by directly computing the Lagrange multiplier from the equation of motion~\eqref{eq:eq_of_motion}.

The equation of motion are thus
\begin{subequations}
\label{eq:motionqv}
\begin{align}
	\dot\xx &= \kker(\zz)\vv \\
	\dot\vv &= \kker(\zz)\trans \Fx(\xx)
\end{align}
where $\zz$ is a solution of the independent Lagrangian system (the driving system in \autoref{def:nonholcoupledsys})
\begin{align}\label{eq:driver}
\ddot\zz +  \smallV'(\zz) = 0 
.
\end{align}
\label{eq:reduced_coupled}
\end{subequations}
Thus, every nonholonomically coupled system can be reduced to a system of ordinary differential equations of the form~\eqref{eq:reduced_coupled}, with first integrals given by the \emph{passenger energy}
\begin{equation}\label{eq:passenger_energy}
	E(\xx,\vv) \coloneqq \frac{1}{2}\vv^2 + U(\xx)
\end{equation}
and the \emph{driver energy}
\begin{equation}\label{small_energy}
	h(\zz,\vz) = \frac{1}{2}\vz^2 + \smallV(\zz).
\end{equation}
Notice that the total energy \eqref{eq:coupled_tot_energy} is the sum of the driver and passenger energies.


\subsection{Quadratic potentials}
\label{sec:quadpot}

\newcommand*\Stiff{K}

We now assume that the potential is quadratic, i.e.,
\begin{align}
U(\xx) = \frac{1}{2} \xx\trans \Stiff  \xx - \FF\trans \xx
,
\end{align}
where $\Stiff$ is a symmetric positive semi-definite matrix and $\FF\in\RR^{n-1}$ is a constant vector.
The corresponding force $\Fx(\xx) = -\frac{\partial U}{\partial \xx}$ is given by
\begin{align}
\Fx(\xx) =  -\Stiff \xx + \FF
.
\end{align}

Using the spectral decomposition, removing eigenvectors corresponding to zero eigenvalues, the matrix $\Stiff$ can be factorised as \[\Stiff = \hstiff \trans \hstiff,\] for a rectangular $m\times (n-1)$ matrix $\hstiff$ of full rank.
We define
\begin{align}
\yy &\coloneqq \hstiff \xx
,
\\
\kstiff(\zz) &\coloneqq  \hstiff \kker(\zz)
,
\end{align}
%
and the projection
\begin{equation}
	\label{eq:projection}
	\pi(\xx,\vv,\zz,\vz) \coloneqq (\hstiff \xx, \vv,\zz,\vz) =  (\yy,\vv,\zz,\vz)
	.
\end{equation}

From \eqref{eq:reduced_coupled} we have $\dot\vv = {\kker(\zz)\trans} \paren{-\hstiff\trans\hstiff \xx + \FF}$, so we get $\dot\vv = -\kstiff(\zz)\trans \yy  + \kker(\zz)\trans \FF$.
The projection $\pi$ therefore intertwines the original system \eqref{eq:motionqv} with the \emph{reduced system}
\begin{gather}
	\label{eq:reduced}
	\begin{aligned}
		\dot\yy &= \kstiff(\zz) \vv \\
		\dot\vv &= -\kstiff(\zz)\trans \yy + \kker(\zz)\trans {\FF}
	\end{aligned}
\end{gather}
where, again, $\zz$ fulfills equation \eqref{eq:driver}.
There is thus a stack of three systems above one another, summarised by the following chain of projections:
\begin{equation}
  (\xx,\vv,\zz,\vz) \longmapsto (\yy,\vv,\zz,\vz) \longmapsto (\zz,\vz).
\end{equation}

The system \eqref{eq:reduced} can be written in matrix notations, using an auxilliary variable $\varepsilon$ with initial condition $1$, as
\begin{align}
\label{eq:yve}
\frac{\ud}{\ud t}
\begin{bmatrix}
	 \yy \\ \vv \\ \varepsilon
\end{bmatrix}
=
\begin{bmatrix}
	0 & \kstiff(\zz) & 0 \\
	-\kstiff(\zz)\trans & 0 & \kker(\zz)\trans {\FF}\\
	0 & 0 & 0
\end{bmatrix}
\mul
\begin{bmatrix}
	 \yy \\ \vv \\ \varepsilon
\end{bmatrix}
\end{align}
We observe that the matrix in \eqref{eq:yve} is an element of $\mathfrak{se}(m+1)$, the Lie algebra of the semidirect product Lie group $\SO(m+1) \ltimes \RR^{m+1}$, where $m$ is the rank of $\hstiff$.
If $\FF$ is zero, the group reduces to $\SO(m+1)$. If $m = 0$, the group reduces to $\RR$.
The Lie algebra structure of equation~\eqref{eq:yve} is central for the reversibility analysis in \autoref{sec:reversibility}.


%

\newcommand*{\Uz}{V}

\subsection{Examples}
\label{sec:nonholexamples}

Here we give several examples of nonholonomic systems of the form presented in \autoref{sec:quadpot}.
The standard form of these problems are used in the literature as test problems for nonholonomic integrators. 
We also suggest new modifications of the standard systems,
constructed so that they fail to be reversible integrable (as detailed in \autoref{sec:reversibility}).

A summary of the problems in this section in terms of the symbols in \autoref{sec:nonholcoupled} is given in \autoref{tab:summary}.

\begin{table}
\centering
\resizebox{\textwidth}{!}{%
  \begin{tabular}{r|cccccc}
    \toprule
      & CVT & Pendulum CVT & Particle & Knife edge & Mobile robot & Rolling disk \\
   & \autoref{sec:cvt} & \autoref{sec:cvt} & \autoref{sec:cvt} & \autoref{sec:knifeedge} & \autoref{sec:verticalrollingdisk} & \autoref{sec:verticalrollingdisk} \\
      \midrule
  $\xx\in$ & $\RR^2$ & $\RR^2$ & $\RR^2$ & $\RR^2$ & $\RR^2\times \RR/\ZZ$ & $\RR^2\times \RR/\ZZ$ \\
  $\zz\in$ & $\RR$ & $\RR/\ZZ$ & $\RR$ & $\RR/\ZZ$ & $\RR/\ZZ$ & $\RR/\ZZ$ \\
  Group & $\SO(3)$ & $\SO(3)$ & $\SO(2)$ & $\RR$ & $1$ & $1$ \\
$\smallV(\zz)$ & $\zz^2/2$ & $\cos\zz - \varepsilon \sin(2\zz)/2$ & $\zz^2/2$ & $0$ & $\sin\zz$ & $0$ \\
$\consmat(\zz)$ & $\begin{bmatrix}1 \\ \zz\end{bmatrix}^\top$  & $\begin{bmatrix}1 \\ \sin\zz\end{bmatrix}^\top$  & $\begin{bmatrix}1 \\ \zz\end{bmatrix}^\top$  & $\begin{bmatrix}-\sin\zz \\ \cos\zz\end{bmatrix}^\top$  & $\begin{bmatrix} 1&0 \\ 0&1 \\ \cos(\zz) & \sin(\zz)\end{bmatrix}^\top$  & $\begin{bmatrix} 1&0 \\ 0&1 \\ \cos(\zz) & \sin(\zz)\end{bmatrix}^\top$ \\
$\hstiff$ & $\begin{bmatrix}1 & 0\\0 &1 \end{bmatrix}$ & $\begin{bmatrix}1 & 0\\0 &1 \end{bmatrix}$ & $\begin{bmatrix}0 &1 \end{bmatrix}$ & $0$ & $0$ & $0$ \\
$\FF$ & 0 & 0 & 0 & (0,-1) & 0 & 0 \\
$F$ & 0 & 0 & 1 & 2 & 2 & 2 \\
$I$ & 3 & 3 & 2 & 2 & 2 & 2 \\
$\theta$ & 2 & 2 & 2 & 1 & 1 & 1 \\
$\rho$ & $(\yy, -\vv)$ & $(\yy, -\vv)$ & $(\yy, -\vv)$ & $\vv$ & $\vv$ & $\vv$ \\
    \bottomrule
  \end{tabular}
}
  \caption{
Summary of the nonholonomically coupled systems presented in \autoref{sec:nonholcoupled}.
$F$ is the dimension of the kernel of $\hstiff$, and hence the dimension of the fibres of the map $\pi$ defined by equation \eqref{eq:projection}.
$I$ is the number of first integrals.
$\theta$ is the number of angle variables.
$\rho$ is the map used to define the reversibility map in \eqref{eq:Rdef}.
}
\label{tab:summary}
\end{table}


\subsubsection{CVT and nonholonomic particle}
\label{sec:cvt}

The \emph{continuously variable transmission} (CVT) problem is a family of coupled nonholonomic system of the form in \autoref{sec:quadpot} with
\begin{equation}
	\FF = 0\\
\end{equation}
and
\begin{equation}
	\hstiff = \begin{bmatrix}
		1 & 0 \\ 0 &1
	\end{bmatrix}
  .
\end{equation}
It is a simple model for a variable transmission gearbox, as illustrated in \autoref{fig:cvt}.
The driving system determines the gear ratio.
We consider different driver systems.

First, the harmonically driven CVT.
In this case, the driver is a harmonic oscillator, in other words, $\smallV(\zz) = \zz^2/2$. 
The nonholonomic constraint is then 
\begin{equation}
	\consmat(\zz) = \begin{bmatrix}
		1 & \zz
	\end{bmatrix}
  .
\end{equation}
Under the name \emph{contact oscillator} this case is considered as a test problem in~\cite{McPe2006}. 
The system is shown to be reversible integrable in~\cite{MoVe2014}; together with reversible KAM theory this gives a theoretical explanation of the excellent numerical results observed in~\cite{McPe2006}.
\begin{figure}
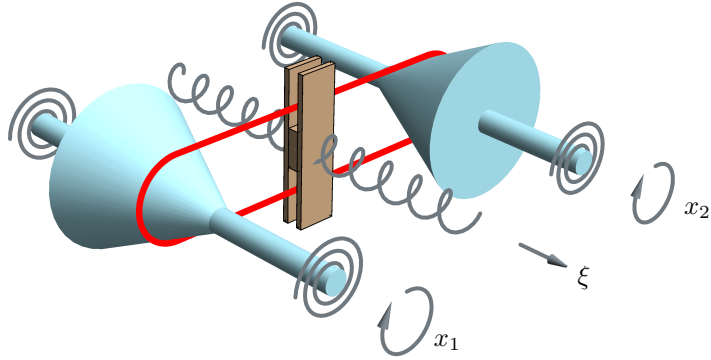

	\centering
	\includestandalone{fig_cvt}
	\caption{Illustration of the principle of the continuous variable transmission (CVT) gearbox. 
	The driving subsystem $(\zz,\vz)$ determines the location of the belt which in turn determines the gear ration between the shafts.
	A nonholonomic system describing the motion is given in \autoref{sec:cvt}.}
\label{fig:cvt}
\end{figure}

The vector spanning the kernel of $\consmat(\zz)$ is
\begin{align}
	\kker(\zz) = \frac{1}{\sqrt{1+\zz^2}}\begin{bmatrix}-\zz \\ 1\end{bmatrix}
  .
\end{align}
Since $\FF=0$ and $\hstiff$ is the identity, 
the evolution matrix~\eqref{eq:yve} is in this case
\begin{align}
\begin{bmatrix}
	 0 &\kker(\zz)\\
	-\kker(\zz)\trans& 0
\end{bmatrix}
,
\end{align}
so the matrix subalgebra is $\so(3)$.

The second type of driver system is the pendulum driven CVT.
Here, the gear ratio (driving system) is governed by a nonlinear pendulum instead of the harmonic oscillator.
The potential for the nonlinear pendulum is now
\begin{equation}\label{eq:pendulum_pot}
  \smallV(\zz) = \cos(\zz) - \varepsilon\sin(2\zz)/2
,
\end{equation} 
where $\varepsilon$ is an arbitrary perturbation parameter,
and the nonholonomic constraint is now given by
  \begin{equation}
	\consmat(\zz) = \begin{bmatrix}
		1 & \sin(\zz)
	\end{bmatrix}
  .
\end{equation}
The vector spanning the kernel of $\consmat(\zz)$ is
\begin{align}
	\kker(\zz) = \frac{1}{\sqrt{1+\sin\paren{\zz}^2}}\begin{bmatrix}-\sin\paren{\zz} \\ 1\end{bmatrix}
  .
\end{align}

We refer to the CVT problem with the driving potential~\eqref{eq:pendulum_pot} as the \emph{pendulum driven} CVT.
As we shall see in \autoref{sec:excvt}, $\varepsilon \neq 0$ corresponds to a non-reversible perturbation, which tends to destroy the good long-time behavior of reversible integrators.

For convenience, the equations of motion are
\begin{equation}\label{eq:pendulum_driven_cvt}
	\begin{split}
		\ddot x_1 &= - x_1 + \lambda \\
		\ddot x_2 &= - x_2 + \sin(\zz)\lambda \\
		\ddot \zz &= - \nabla \smallV(\zz) \\
		0 &= \dot x_1 + \sin(\zz) \dot x_2 .
	\end{split}
\end{equation}

The last type of driver system is the \emph{nonholonomic particle}, considered in~\cite{CoMa2001}.
It is a degenerate case of harmonically driven CVT, where the spring of one of the shafts has zero stiffness, so
\begin{equation}
	\hstiff = \begin{bmatrix}
		 0&1
	\end{bmatrix} 
.
\end{equation}
This gives
\begin{align}
	\kstiff(\zz) = \hstiff \kker(\zz) = \frac{1}{1+\zz^2}
  .
\end{align}
Thus, the evolution matrix~\eqref{eq:yve} is
\begin{align}
\begin{bmatrix}
	 0 &\kstiff(\zz)\\
	-\kstiff(\zz)\trans& 0
\end{bmatrix}
  ,
\end{align}
so the Lie subalgebra is $\g = \so(2)$.



\subsubsection{Knife edge}
\label{sec:knifeedge}

\begin{figure}
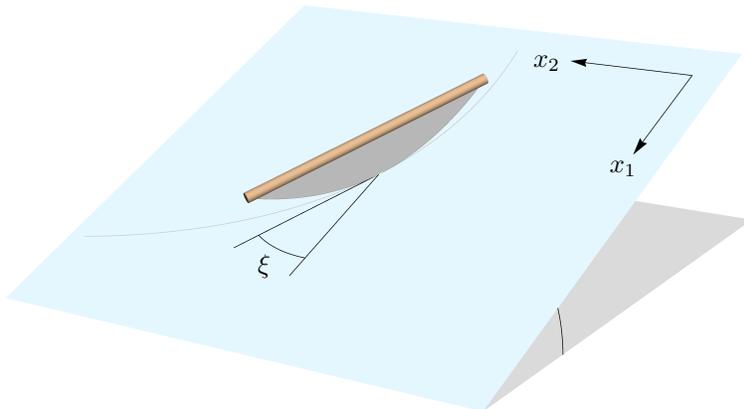

	\centering
	\includestandalone{fig_knife_edge}
	\caption{Illustration of the knife edge system~\eqref{eq:knife_edge}.
	The contact point of the knife edge, or skate, is sliding under gravity on the inclined plane.
	The direction of sliding is determined by the angle~$\zz$; one may think of a ``one-legged skater'', changing direction of his skate according to the driving system.
	}
	\label{fig:knife_edge}
\end{figure}

The \emph{knife edge} (as denoted in~\cite[\S\,1.6]{Bl2015}), or \emph{skate on an inclined plane} (as denoted in~\cite[\S\,9.1]{RaRh2000} and \cite[\S\,1.2.5]{ArKoNe2006}), is given by
\begin{equation}\label{eq:knife_edge}
	\begin{split}
		\ddot x_1 &= 1 -\lambda \sin(\zz) \\
		\ddot x_2 &= \lambda \cos(\zz) \\
		\ddot \zz &= 0 \\
		0 &= -\dot x_1\sin(\zz) + \dot x_2\cos(\zz) .
	\end{split}
\end{equation}
An illustration is given in \autoref{fig:knife_edge}.
In terms of the data in \autoref{sec:nonholcoupled}, the system is defined by
\begin{equation}
\begin{aligned}
	\FF &= \begin{bmatrix}1 \\ 0\end{bmatrix}\\
	\smallV(\zz) &= 0\\
	\consmat(\zz) &= \begin{bmatrix}
		-\sin(\zz) & \cos(\zz)
	\end{bmatrix}\\
\end{aligned}
\end{equation}
and $\hstiff=0$ (it is a $0\times 2$ matrix), so $\yy= [\,]$ (the empty vector).
The kernel vector $\kker(\zz)$ is
\begin{align}
  \label{eq:kinfeedgekernel}
	\kker(\zz) = \begin{bmatrix}\cos(\zz) \\\sin(\zz)\end{bmatrix}
  ,
\end{align}
and the evolution equation of the reduced system~\eqref{eq:reduced} is simply
\begin{equation}
  \dot\vv = \kker(\zz)\trans \FF = \cos(\zz)
  .
\end{equation}
Since $\yy =[\,]$, the matrix in \eqref{eq:yve} is given by
\begin{align}
\begin{bmatrix}
	0&\cos(\zz)\\
	0&0
\end{bmatrix}
     ,
\end{align}
so the underlying group is~$\RR$.

Consider now a slightly perturbed version of the knife edge, where 
\begin{equation}
	\consmat(\zz) = \begin{bmatrix}
		-\sin(\zz) & \cos(\zz) - \varepsilon 
	\end{bmatrix}.
\end{equation}
When $\varepsilon = 0$ this is exactly the knife edge.
As we shall see in \autoref{sec:exknife}, $\varepsilon \neq 0$ implies non-integrability.

\subsubsection{Vertical rolling disk and mobile robot}
\label{sec:verticalrollingdisk}

\begin{figure}
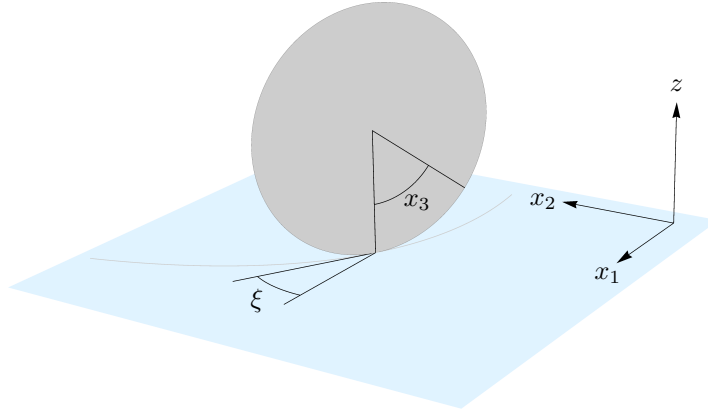

	\centering
	\includestandalone{fig_vert_disk}
	\caption{Illustration of the vertical disk, or rolling penny, given by~\eqref{eq:vert_disk}.
	The rotation of the penny is described by~$x_3$ (measured from the $z$-axis) and the position of the contact point by $(x_1,x_2)$.
	The directional angle is described by~$\zz$.
	Because of conservation of angular momentum, we have $\ddot\zz = 0$.
	}
	\label{fig:vert_disk}
\end{figure}

The \emph{vertical rolling disk} is a standard example of a nonholonomic system \cite[\S\,1.4]{Bl2015}.
It is given by
\begin{align}\label{eq:vert_disk}
	\ddot x_1 &= \lambda_1 \\
	\ddot x_2 &= \lambda_2 \\
	\ddot x_3 &= \lambda_1 \cos(\zz) + \lambda_2 \sin(\zz) \\
	\ddot\zz &= 0 \\
	0 &= \dot x_1 + \cos(\zz)\dot x_3 \\
	0 &= \dot x_2 + \sin(\zz)\dot x_3.
\end{align}
An illustration is given in \autoref{fig:vert_disk}.
In terms of \autoref{sec:nonholcoupled}, the data are
\begin{equation}
	\begin{aligned}
		\hstiff &= 0\\
		\FF &= 0\\
	\consmat(\zz) &= \begin{bmatrix}
		1& 0 &\cos(\zz)\\
		0 &1&\sin(\zz)
	\end{bmatrix}
	\\
	\smallV(\zz) &= 0
	\end{aligned}
\end{equation}
so the kernel $\kker(\zz)$ is given by
\begin{align}
	\kker(\zz) = \frac{1}{\sqrt{2}}\begin{bmatrix}-\cos(\zz)\\-\sin(\zz)\\1\end{bmatrix}
  .
\end{align}
Since $\yy$ is the empty vector and $\FF=0$, the reduced equation~\eqref{eq:reduced} is simply $\dot\vv=0$,
so the underlying group is the trivial group $\one$.

A modification of the vertical rolling disk is the \emph{mobile robot} \cite{CoMa2001}, which corresponds to
\begin{align}
	\smallV(\zz) = \sin(\zz).
\end{align}
Thus, the driving system for the mobile robot is
\begin{equation}
	\ddot\zz = -\cos(\zz).
\end{equation}
Everything else is identical to the vertical rolling disk.

\section{Numerical experiments}\label{sec:numerical_examples}

In this section we give examples and counter-examples of good long-time behaviour for nonholonomic integrators applied to the test problems of \autoref{sec:nonholexamples}.
The counter-examples stem for the perturbed versions of the knife edge and the CVT.
The two perturbed problems correspond to two different mechanisms destroying the good long-time behaviour of nonholonomic integrators:
(i) by removal of the integrable structure (perturbed knife edge), and (ii) by removal of the reversible structure (perturbed CVT).

As representatives for nonholonomic integrators we use five different methods:
\begin{enumerate}
	\item DLA${}^\alpha$ suggested by \citet{CoMa2001}, with $\alpha = 1/2$.
	Since $\alpha = 1/2$, the resulting integrator is reversible.
	The method is given by equation~\eqref{eq:dla_half} in appendix \autoref{sub:dla}.
	\item DLA${}^\alpha$ with $\alpha = 0.4$, making it a non-reversible integrator.
	The method is given by equation~\eqref{eq:dla_alpha} in appendix \autoref{sub:dla}.
	\item DLA${}^{0,1}$ suggested by \citet{McPe2006}.
	The method is given by equation~\eqref{eq:dla_01} in appendix \autoref{sub:dla}.
	\item The leap-frog method (LF) for nonholonomic systems, suggested by \cite{McPe2006} and later revisited by \citet*{FeIgMa2008}.
	The method is given by equation~\eqref{eq:leap_frog} in appendix \autoref{sub:leapfrog}.
	\item The discrete derivative method (DD), initially suggested for Hamiltonian systems by \citet{Go1996} and later adopted to nonholonomic systems by Betsch~\cite{Be2006}. 
    The method is given in \autoref{sec:epinteg}.
\end{enumerate}

\subsection{Knife edge}
\label{sec:exknife}
The initial data is
\begin{equation}
	x_1(0) = 0, \quad
	x_2(0) = 0, \quad
	\zz(0) = \pi/2, \quad
	\dot x_1(0) = 0, \quad
	\dot x_2(0) = 0, \quad
	\vz(0) = 1.
\end{equation}
This corresponds to an initially horizontal skate which is rotated by the driving system thereby picking up speed in the direction of the skate due to gravity.
The unperturbed ($\varepsilon=0$) and perturbed ($\varepsilon=0.1$) systems are integrated using 5 methods:
\begin{center}
DLA${}^{0.5}$, DLA${}^{0.4}$, DLA${}^{0,1}$, DD, and LF.	
\end{center}
The stepsize for all methods is $\Delta t=\pi/10$.
The integration interval is $[0,100]$.

Notice that the dynamics of the driver system $(\zz,\vz)$ is trivial for the knife edge (simply $\ddot \zz = 0$).
Thus, all the integrators provide the exact solution of the driver system (by integrator consistency).
Consequently, all the integrators exactly preserve the energy $h$ of the driving system.

The evolution of the energy error $\abs{H(t)-H(0)}$ for all 5 methods is given in \autoref{fig:energy_knife_edge}.
We make the following observations.
For the unperturbed system, all integrators except DLA${}^{0.4}$ exactly or nearly preserves the energy integral $H$.
For the perturbed system, all integrators except DD give a drift in the total energy $H$.


\begin{figure}
  \centering
	\includegraphics[width=0.99\textwidth]{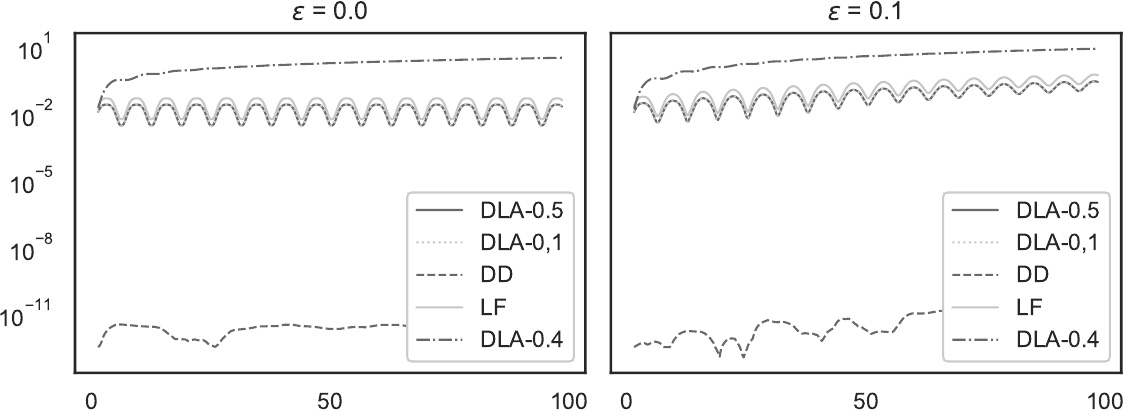} 
	\caption{Evolution of the error $\abs{H-H(0)}$ of the energy integral for 5 methods applied to the knife edge. 
	The data are convoluted by a running mean with a window of about 3 time units.
	For all methods except DLA${}^{0.4}$, the error for the unperturbed system ($\varepsilon=0$) is bounded in time.
	For the perturned system ($\varepsilon=0.1$) all methods except DD show energy drift.
	}
	\label{fig:energy_knife_edge}
\end{figure}


\subsection{Continuously variable transmission (CVT)}
\label{sec:excvt}

The system here is the CVT with potential for the driver system given by \eqref{eq:pendulum_pot}.
We consider two different sets of initial data.
First,
\begin{equation}
	x_1(0) = 1, \quad
	x_2(0) = 1, \quad
	\zz(0) = 0, \quad
	\dot x_1(0) = 0, \quad
	\dot x_2(0) = 0, \quad
	\vz(0) = 1.8973666
\end{equation}
corresponding to total energy $H_0 = 2.8$.
Second,
\begin{equation}
	x_1(0) = 1, \quad
	x_2(0) = 1, \quad
	\zz(0) = 0, \quad
	\dot x_1(0) = 0, \quad
	\dot x_2(0) = 0, \quad
	\vz(0) = 2.82842712
\end{equation}
corresponding to total energy $H_0 = 5.0$.
The two different sets of initial data yield two different types of behaviour for the driver system.
When the energy level is low ($H_0 = 2.8$) the phase diagram of the driver system corresponds to a nonlinear pendulum going back and forth: we call this the \emph{oscillating driver}.
When the energy level is high ($H_0 = 5.0$) the phase diagram of the driver system corresponds to a nonlinear pendulum with enough kinetic energy so that it does not turn back, but keep on rotating in the same direction: we call this the \emph{rotating driver}.
The setup is illustrated in \autoref{fig:phase_small_cvt}.

\begin{figure}
	\centering
	\includegraphics[width=0.99\textwidth]{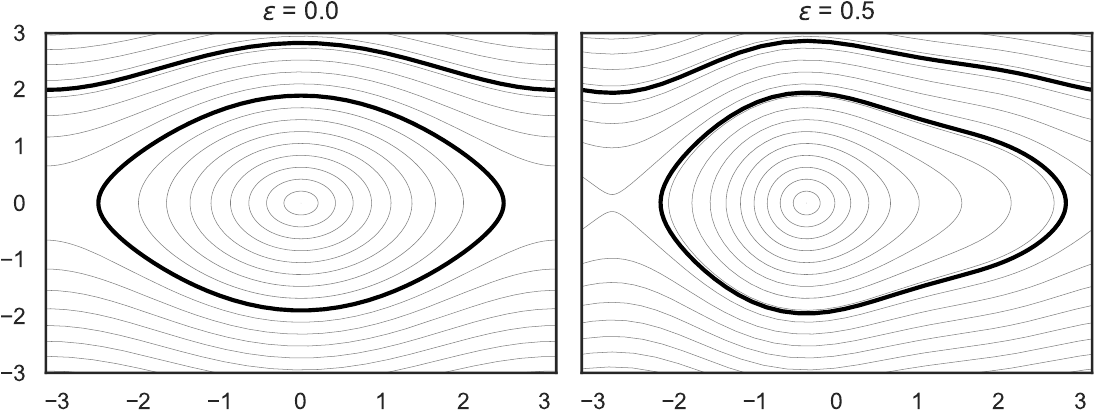}
	\caption{Phase diagrams for the $(\zz,\vz)$ subsystem of the CVT problem, with $\varepsilon=0$ (left) and $\varepsilon=0.5$ (right). 
	The circular (oscillating driver) and upper (rotating driver) paths correspond, respectively, to the low ($H_0 = 2.8$) and high ($H_0=5.0$) energy levels. 
	Both diagrams are symmetric under the standard reversibility map $\vz \mapsto -\vz$.
	Notice that the left diagram also is symmetric under the `non-physical' reversibility map $\zz\mapsto -\zz$, whereas the right diagram does not have this symmetry (due to the perturbation).
	}
	\label{fig:phase_small_cvt}
\end{figure}

The unperturbed ($\varepsilon=0$) and perturbed ($\varepsilon=0.1$) CVT systems, for both choices of initial data, are integrated using 5 methods:
\begin{center}
DLA${}^{0.5}$, DLA${}^{0.4}$, DLA${}^{0,1}$, LF, and DD.
\end{center}
The stepsize for all methods is $\Delta t=0.1$.
The integration interval is $[0,3000]$.

The evolution of the passenger energy error $\abs{E(t)-E(0)}$ for all 5 methods is given in \autoref{fig:energy_cvt}.
The evolution of the driver energy error $\abs{h(t)-h(0)}$ for all 5 methods is given in \autoref{fig:small_energy_cvt}.
Since the CVT system is integrable (as detailed in \autoref{sec:reversibility} and \autoref{sec:fibrations_over_int_sys}), there is an additional integral that is not available explicitly (see \autoref{prop:integsys}).
Although an explicit formula is not available, we can study this integral at the Poincar\'e section given by sampling every period of the driver system.
It can be interpreted as the latitude along a certain direction (depending on the initial data) of the vector $(x_1,x_2,v)$ with $v$ given by \eqref{eq:diffeqv}; we therefore call it the \emph{latitude integral}.
The evolution of the latitude integral error (at the Poincar\'e section) for all 5 methods is given in \autoref{fig:latitude_cvt}.

\begin{remark}
	Since the first preprint of this paper, the asymmetrically perturbed CVT problem has been used as a test problem for energy-momentum type integrators \cite{CeFaHoMa2019}.
	In that paper, however, the authors evaluate the performance solely based on conservation of energy. 
	This misses the point; to say if an integrators performs well or not on this problem one has to study all the integrals, in particular the latitude integral (for which there is no simple formula).
\end{remark}


\begin{figure}
  \centering
	\includegraphics[width=0.99\textwidth]{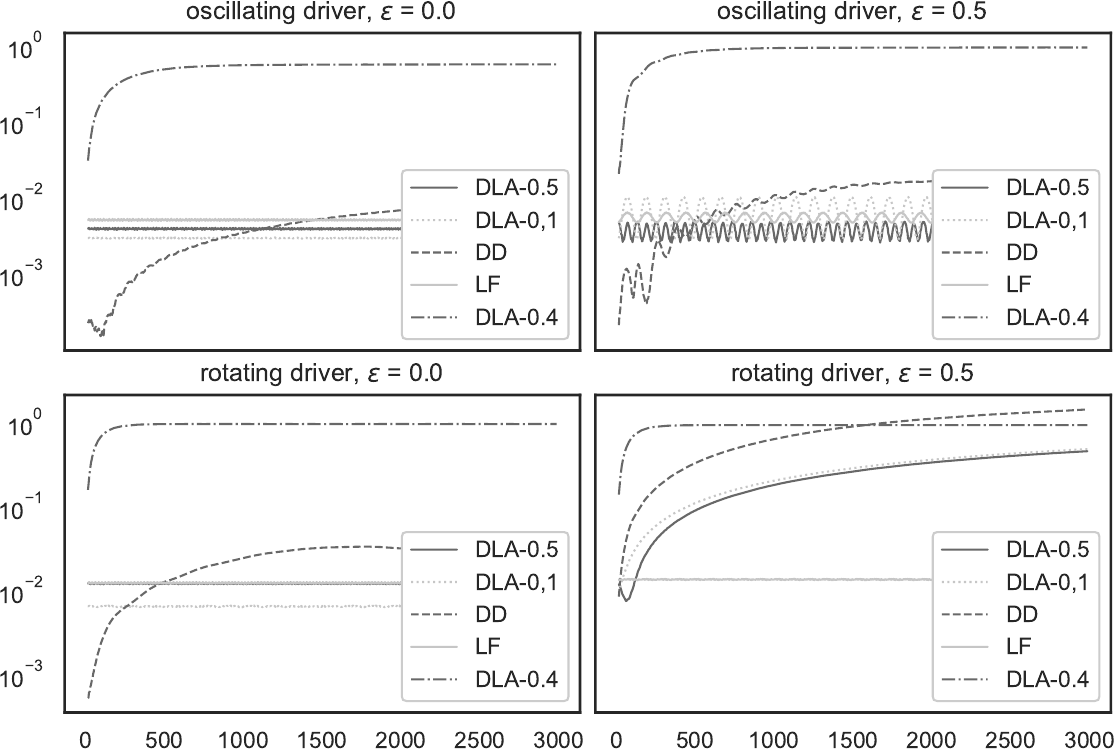}
	\caption{
	Error in the passenger energy \eqref{eq:passenger_energy} for the 5 methods applied to the pendulum driven CVT problem~\eqref{eq:pendulum_driven_cvt} for both the oscillating and rotating driver, and two different values of~$\varepsilon$.
	The data are convoluted by a running mean with a window of 30 time units.
	In the bottom right diagram, a systematic drift of the energy occurs for all methods but LF.
	}
	\label{fig:energy_cvt}
\end{figure}

\begin{figure}
  \centering 
	\includegraphics[width=0.99\textwidth]{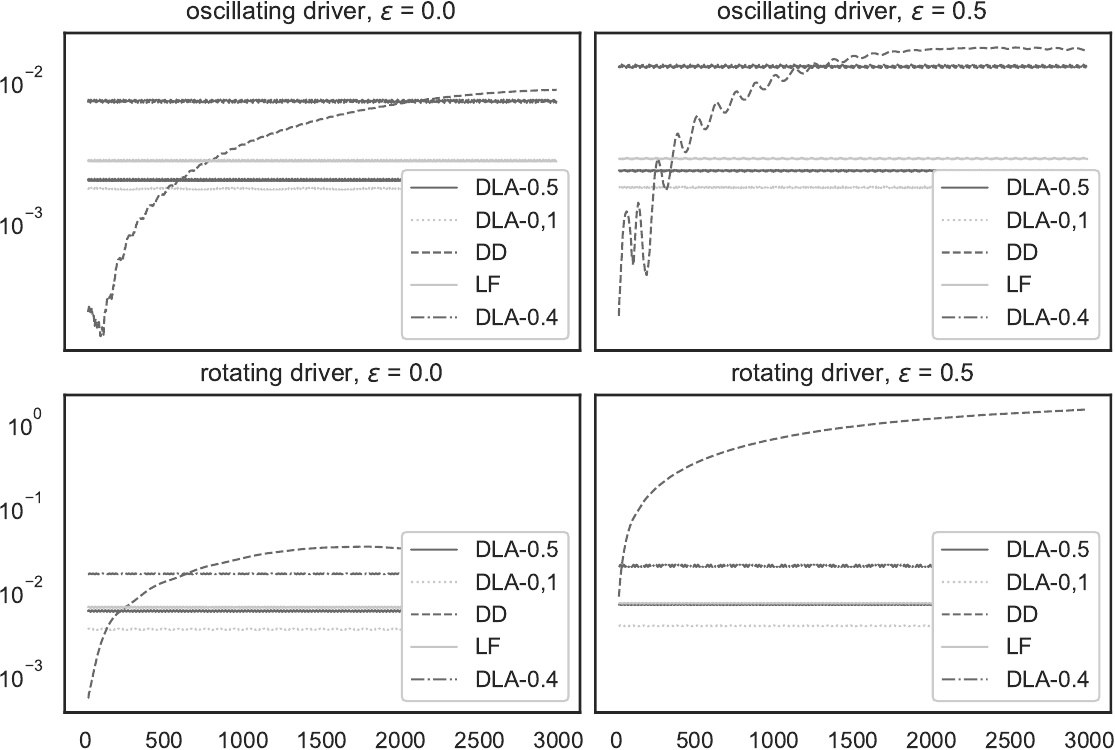}
	\caption{
	Error in the driver energy \eqref{small_energy} for the 5 methods applied to the pendulum driven CVT problem~\eqref{eq:pendulum_driven_cvt} for both the oscillating and rotating driver, and two different values of~$\varepsilon$.
	The data are convoluted by a running mean with a window of 30 time units.
	The only method that does not nearly conserve the driver energy is DD.
	Notice that even DLA${}^{0.4}$ gives a good behaviour, despite not being reversible: this is fully explained by the fact that the driver system in itself is a Hamiltonian system and DLA${}^{0.4}$ is a symplectic integrator in absence of constraints, so backward error analysis of symplectic integrators (cf.~\cite[\S\,IX.3]{HaLuWa2006}) fully explains the near conservation.
	}
	\label{fig:small_energy_cvt}
\end{figure}

\begin{figure}
  \centering 
	\includegraphics[width=0.99\textwidth]{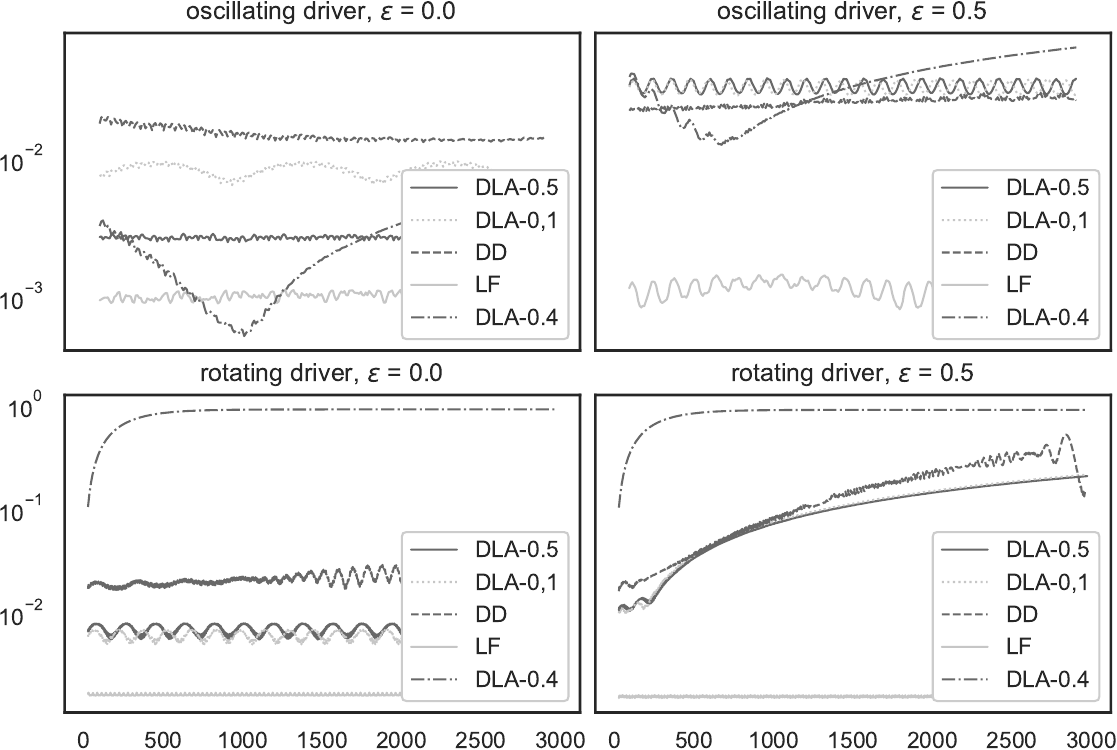}
	\caption{
	Error in the latitude integral corresponding to the extra first integral due to integrable structure of the CVT problem (see \autoref{prop:integsys}).
	This first integral is `hidden': we have no explicit formula for it.
	We can only compute it on a Poincar\'e section by sampling every period of the driver system.
	The data is convoluted by a running mean with a window of 30 time units.
	If the numerical integrator is reversible, then preservation of this integral is fully explained by reversible KAM theory, as discussed in \autoref{sec:KAM}.
	Notice that this explains all but one of the results:
	For the oscillating driver, where reversibility in the driver system is still preserved (corresponding to the green curve in \autoref{fig:phase_small_cvt}), or when $\varepsilon = 0$, all reversible integrators exhibit no drift.
	For the rotating driver with $\varepsilon\neq 0$, every integrator fails except the leap-frog method.
	It is an open problem to explain why the leap-frog method works for that particular system (note how leap-frog exhibits drift for the perturbed knife edge problem, see \autoref{fig:energy_knife_edge}).
	}
	\label{fig:latitude_cvt}
\end{figure}

\subsection{Mechanism for near conservation of integrals}
\label{sec:KAM}


The perturbation theory of Kolomogorov, Arnold, and Moser (KAM theory) comes in two flavours: symplectic and reversible.
In short, it states that Hamiltonian/reversible perturbations of Hamiltonian/reversible integrable systems nearly preserve invariant tori.
In combination with backward error analysis of numerical integrators, KAM theory rigorously explains near conservation of first integrals for symplectic/reversible integrators applied to Hamiltonian/reversible integrable systems.
For a thorough treatment we refer to the monograph by \citet{HaLuWa2006} and references therein.

KAM theory, however, is not readily applicable to nonholonomic systems: it applies to integrable systems of ODE.
A main result of this paper is that the class of nonholonomically coupled systems introduced in \autoref{sec:nonholcoupled} can be made compatible with the setting of reversible KAM theory for backward error analysis of reversible integrators, provided that the nonholonomic systems and integrators are reversible with respect to the same reversibility map.
We now state this result, which in turn relies on results proved in \autoref{sec:reversibility} and \autoref{sec:fibrations_over_int_sys}.




As always when analyzing constrained systems, the first step is to reduce the constrained system to an ordinary differential equation on a state space manifold $\Upstairs$. 
Details of how this is done for the nonholonomically coupled systems are given in \autoref{sec:reversibility} below.

\begin{definition}
\label{def:fibrevinteg}
  Consider a vector field $X$ on $\Upstairs$.
We assume that there is a surjection $\pi \colon \Upstairs \to \Downstairs$.
(In our case, $\Upstairs$ and $\Downstairs$ are vector spaces or cylinders, and this surjection is just a linear map.)
	Assume that $X$ descends to a vector field $Y$, i.e., $\pi_* X = Y$.
  We now say that $X$ is \emph{fibrated over an $R$-integrable system} if there is
	 an involution $R$ defined on $\Downstairs$, and the system $Y$ is $R$-integrable (in the precise sense of \autoref{def:Rintegrable} below).
\end{definition}

\begin{theorem}
\label{prop:fibrevinteg}
	The state space formulation of each unperturbed ($\varepsilon=0$) coupled nonholonomic system defined in \autoref{sec:nonholexamples} is fibrated over a reversible integrable system (\autoref{def:fibrevinteg}).
\end{theorem}
\begin{proof}
	The theorem is a consequence of the results presented in \autoref{sec:reversibility} and \autoref{sec:fibrations_over_int_sys} below, in particular the final result \autoref{prop:redrevinteg}. 
\end{proof}





We now describe the mechanism by which reversible KAM and backward error analysis can be adopted to ODEs fibrated over integrable systems.
Given a nonholonomic system whose state space formulation is fibrated over a reversible integrable system, with a projection $\pi$ and a reversibility map (an involution) $R\colon \Downstairs\to\Downstairs$ as in \autoref{def:fibrevinteg}, 
suppose that a nonholonomic integrator $\Phi_h\colon \Upstairs\to\Upstairs$ is compatible with $\pi$ and $R$ in the following sense:
	\begin{enumerate}[label=\upshape(\roman*)]
		\item\label{cond:descending}
			It descends to a method $\Psi_h\colon\Downstairs\to\Downstairs$, i.e.,
      \begin{equation}
        \Psi_h \circ \pi = \pi \circ \Phi_h
        .
      \end{equation}
\item\label{cond:reversible}
	The descending integrator $\Psi_h$ preserves the reversibility structure $R$, i.e.,
		\begin{equation}
			\label{eq:methodreversible}
	R \circ \Psi_h \circ R = \Psi_h\inv
  .
		\end{equation}
\end{enumerate}
Thus, $\Psi_h$ is a reversible integrator for the reversible integrable system on $\Downstairs$ corresponding to the vector field $Y$.
By backward error analysis of numerical integrators on manifolds (see \cite[Ch.~IX.2]{HaLuWa2006} or \cite{Ha2011b}) we then get, up to exponentially small terms and for exponentially long times, that the integrator $\Psi_h$ corresponds to the exact flow of a reversible vector field $Y_h$.
By the reversible KAM Theorem~\cite[Th.~2]{Se1998}, this explains the near conservation of the first integrals for the reversible system on $\Downstairs$.
Since these integrals are unaffected by the fibre motion in the fibration $\Upstairs\to\Downstairs$, and since the numerical integrator preserves the fibration, they are also nearly conserved as first integrals of the system on $\Upstairs$.
This explains near conservation of integrals for reversible perturbations of reversible integrable nonholonomic systems.
The final step is to verify that the conditions~\ref{cond:descending} and~\ref{cond:reversible} above for the integrators in \autoref{sec:integrators}.
This is straightforward.
Consider, for example, the knife edge example (\autoref{sec:knifeedge}).
In this case, the projection $\pi$ is just $\pi(\xx,\vv,\zz,\vz) = (\vv,\zz,\vz)$.
The condition~\ref{cond:descending} is true for all methods in \autoref{sec:integrators} except DD (the discrete derivative method).
Next, the reversibility map is $\rho(\vv,\zz,\vz) = (\vv,\zz,-\vz)$. 
A direct calculations shows that condition~\ref{cond:reversible} holds for all \DLA\ methods, except \DLA{\alpha} with $\alpha \neq \frac{1}{2}$.

To summarize our findings: if a nonholonomic integrator is compatible with the fibrated integrable structure of a nonholonomic coupled system, and at the same time respects the reversible structure of the integrable system, then reversible KAM theory fully explains the good long time behaviour.
If, however, the nonholonomic integrator fails to preserve either the fibration structure or the reversible structure of the underlying integrable ODE, then one cannot expect good long time behaviour.
The perturbed problems are exactly constructed to break these structures:
The perturbed Knife Edge breaks the fibration over an integrable system, whereas the perturbed CVT for the high energy level, although still fibrated over an integrable system, breaks reversibility.
Let us now discuss how our predictions adhere to the numerical results.

\subsection{Explanation of the numerical results}\label{sub:explanation}

\newcommand{\y}{$\surd$}
\newcommand{\n}{$\times$}

From the previous section we extract the following properties of the 5 nonholonomic integrators used in the examples above:

\begin{center}
	\begin{tabular}{rccccc}
								& DLA${}^{0.5}$ & DLA${}^{0.4}$ & DLA${}^{0,1}$ & DD & LF \\
		Preserves fibration		& \y 			& \y 			& \y 			& \n & \y \\
		Preserves reversibility & \y 			& \n 			& \y 			& \y & \y
	\end{tabular}
\end{center}

Let us now go through the two numerical examples (knife edge and CVT) and relate the performance of each method to our theoretical predictions.

By \autoref{prop:fibrevinteg} the unperturbed ($\varepsilon=0$) knife edge is fibrated over a reversible integrable system.
By our theoretical predictions, every method that preserves both the fibration and the reversibility should perform well.
And indeed they do, as confirmed in the left diagram of \autoref{fig:energy_knife_edge}.
We also see that the non-reversible method, DLA${}^{0.4}$, exhibits drift.
This constitutes strong evidence that reversibility, not the discrete Lagrange--d'Alembert principle, is behind near conservation.
Notice that DD exactly preserves the energy by construction.
Consider now the more interesting case of the perturbed ($\varepsilon=0.1$) knife edge.
Recall that this perturbation is constructed to destroy the fibration structure, so in this case our theoretical results cannot be applied.
In the right diagram of \autoref{fig:energy_knife_edge} we see that all methods, except of course DD, exhibit energy drift.
This shows that reversibility alone is not enough for near conservation; it is more intricate.

We now turn to the unperturbed ($\varepsilon=0$) CVT problem.
Recall that this problem has three first integrals: (i) the driver energy, (ii) the passenger energy, and (ii) the more complicated latitude integral.
Also recall that there are two types of drivers depending on the initial conditions: oscillating and rotation (see \autoref{fig:phase_small_cvt}).
By \autoref{prop:fibrevinteg} the system is fibrated over a reversible integrable system.
As expected from our theory, we see in the left columns of Figures~\ref{fig:energy_cvt}, \ref{fig:small_energy_cvt}, and \ref{fig:latitude_cvt} that the methods that are both fiber preserving and reversible nearly conserve the integrals.
Notice also that DD exhibit drift in all the integrals, although total energy is exactly conserved by construction.
This illustrates that methods that are ``forced'' to preserve one of the integrals generally do not perform well on other integrals.\footnote{There are generalizations of DD that preserves more than just one integral (so called \emph{energy-momentum methods}). We could have used such a method to preserve also the driver energy. However, our point is that hidden integrals, for which explicit formulae are not available, are not conserved by energy-momentum methods.}\
Notice in the left column of \autoref{fig:small_energy_cvt} that the driver energy is nearly conserved for DLA${}^{0.4}$ despite this integrator not being reversible.
This is expected since the driver subsystem is independent from the rest of the dynamics and is actually a Hamiltonian subsystem.
Thus, since DLA${}^{0.4}$ applied to Hamiltonian systems is a symplectic integrator (albeit not reversible), the standard theory of backward error analysis for symplectic integrators explains the near conservation of the driver energy.
(The passenger energy, however, is not conserved for DLA${}^{0.4}$.)

Finally, we turn to the perturbed ($\varepsilon=0.5$) CVT problem.
It is constructed to still be fibrated over an integrable system, but it is no longer reversible with respect to the standard reversibility map (being integrable it is, of course, still reversible with respect to a different, more complicated reversibility map).
Actually, there are two possible reversibility maps, reflection in the $\qq$-plane or the $\pp$-plane, and the reversible methods preserve both of these.
$\varepsilon \neq 0$ destroys the $\qq$-reversibility, but the $\pp$-reversibility remains.
However, when the driver is rotating the $\pp$-reversibility becomes ``invisible'' to the dynamics as illustrated in \autoref{fig:phase_small_cvt}.
Therefore, according to our theoretical predictions, the fibre preserving and reversible methods should perform well with an oscillating driver.
Indeed they do, as can be seen in the upper left diagram of Figures~\ref{fig:energy_cvt}, \ref{fig:small_energy_cvt}, and \ref{fig:latitude_cvt}:
only the non-reversible DLA${}^{0.4}$ and non-descending DD methods exhibit drift.
Since the independent driver subsystem is Hamiltonian, we also see for the rotating driver at the bottom right of \autoref{fig:small_energy_cvt} that all methods that are symplectic when applied to Hamiltonian systems preserve the driver energy, regardless of weather they are reversible or not.
So far, our theory is perfectly aligned with the numerical simulations in ``both directions'':
whenever the theory is applicable we observe near conservation (simply verifying the theoretical results), but also, whenever the theory is not applicable we observe drift.
According to this, we expect for a rotating driver (where our theory does not apply) to see drift in the passenger energy in all the methods.
In the bottom right of \autoref{fig:energy_cvt} we observe drift for all methods but LF.
This fact, that LF conserves the first integrals of a non-reversible nonholonomic problem, came as a surprise.
At this stage we have no explanation for it.
We predict, however, that there is some reversibility map, or some modified symplectic structure, that is shared by that particular problem and the integrator map, and which would allow us to apply KAM theory.

In the remainder of the paper we prove results leading to \autoref{prop:fibrevinteg} above.

\section{Linear, periodic systems: averaging and reversibility}
\label{sec:reversibility}

Systems of the form \eqref{eq:yve} in \autoref{sec:nonholexamples} can be written generally as
\begin{equation}
  \dot{\vect{u}} = \matA(\zz ) \vect{u}
\end{equation}
where $\zz$ is the solution of the driving system~\eqref{eq:driver} and $\matA(\zz)$ takes values in a Lie algebra $\g$.
Assume now that the energy function of the driving system $h(\zz,\vz) = \vz^2/2 + \smallV(\zz)$ has bounded level sets, so that solutions are periodic.
After a change of variables, the system $(\zz,\vz)$ can be rewritten using an angle variable $\theta$ such that $\theta' = \omega(h)$ for some function of the energy $h$ (the action variable).
We are thus led to study systems of the form
\begin{equation}
\label{eq:sysAtheta}
  \dot{\vect{u}} = \matA(\theta) \vect{u}
  ,
\end{equation}
where $\matA$ is periodic in $\theta$.

The first aim of this section is to show that systems of the form~\eqref{eq:sysAtheta} can be transformed into autonomous linear systems by means of \emph{averaging}.
The second aim is to give conditions under which this averaging transformation is a \emph{reversible map}.

The results in this section (\autoref{th:integrability} and \autoref{prop:reversible}) are generalizations of Theorem~3.1 and Theorem~3.2 in \cite{MoVe2014}.

\subsection{Averaging} 

We first study averaging of systems of the form~\eqref{eq:sysAtheta}.
Averaging means that after a reparametrization, the system \eqref{eq:sysAtheta} is equivalent to a system of the same form but with a \emph{constant} matrix $\matAbar$, which can be interpreted as the average of $\matA$.
We first reformulate systems of the form~\eqref{eq:sysAtheta} as follows:

\begin{definition}
\label{def:autoper}
	Let $\g$ be a Lie subalgebra of $\mathfrak{gl}(n,\RR)$. 
  A $\g$-periodic differential equation is a system of the form 
	\begin{equation}
	\label{eq:autoper}
	\begin{split}
		{\vect u}' &= \matA\big(\theta\big)\vect u\\
		{\theta}' &= 1
	\end{split}
	\end{equation}
	where $(\vect u,\theta) \in \RR^n \times \Torus$ and $\matA \colon \Torus \to \g$ is a smooth mapping.
\end{definition}

\begin{theorem}
\label{th:integrability}
Consider a $\g$-periodic system (\autoref{def:autoper}).
	Then there is a smooth change of variables
	\begin{equation}
	\Psi \colon \RR^n \times \Torus \ni (\vect u,\theta)\mapsto (\vect v(\vect u,\theta),\theta) \in \RR^{n} \times \Torus
	\end{equation}
	such that the $\g$-periodic system~\eqref{eq:autoper} expressed in the new variables $(\vect v,\theta)$ takes the form
	\begin{equation}
		\label{eq:averaged}
	\begin{split}
	{\vect v}' &= \matAbar \vect v \\
	{\theta}' &= 1
	\end{split}
	\end{equation}
	for a constant ``average'' element $\matAbar \in \g$.
\end{theorem}


\begin{proof}

One defines the flow map $\flow(\thera)$ as the solution operator of the differential equation defined for all $\thera \in \RR$ by
\begin{equation}\label{eq:nonautonomous}
	\frac{\ud\vect w(\thera)}{\ud\thera} = \matA(\thera)\vect w(\thera)	
\end{equation}
with initial condition at $\thera = 0$ --- the initial time matters because the differential equation is not autonomous.
This means that if $\vect w$ is a solution of~\eqref{eq:nonautonomous}, then
\begin{equation}
\label{eq:defflowmap}
\vect w(\thera) = \flow(\thera) \vect w(0) \qquad \forall \tau \in \RR
,
\end{equation}
and vice versa.

\newcommand*\bPsi{\overline{\Psi}}
Since $\matA(\thera)\in\g$ for all $\thera$, the flow map after one period, i.e., $\perPhi$, is an element of $\Group$.
Let $\Group \subset \GL(n,\RR)$, the connected Lie group corresponding to the Lie algebra $\g$.
As the exponential is surjective from $\g$ to $\Group$, there exists $\matAbar\in\g$ such that
\begin{equation}
\label{eq:defmatAbar}
\perPhi = \exp(\matAbar)
.
\end{equation}
We define the mapping $\bPsi \colon \RR^n \times \RR \to \RR^n \times \RR$ by $\bPsi(\vect u,\thera) = \big(\vect v(\vect u,\thera),\thera\big)$ and
\begin{equation}
  \vect v(\vect u,\thera) := \exp(\matAbar \thera) \flow(\thera)\inv \vect u
.
\end{equation}

Recall that, as $\matA$ is periodic, for any integer $n\in\ZZ$ and any $\tau\in\RR$ we have (cf.\ Floquet theory \cite[\S\!~28]{Ar2006})
\begin{equation}
\label{eq:Floquet}
\flow(n+\tau) = \flow(\tau)\flow(n)
.
\end{equation}
Now, $\vect v(\vect u,\thera)$ is periodic in $\thera$, of period one, because, due to~\eqref{eq:Floquet}, and the definition~\eqref{eq:defmatAbar} of $\matAbar$,
\begin{equation}
\exp\big(\matAbar(\thera+1)\big)\flow(\thera+1)\inv = \exp(\matAbar\thera)\exp(\matAbar)\perPhi\inv\flow(\thera)\inv =  \exp(\matAbar\thera)\flow(\thera)\inv
.
\end{equation}
As a result, the mapping $\bPsi$ induces a mapping $\Psi \colon \RR^n \times \Torus \to \RR^n \times \Torus$.
Note that, as $\bPsi$ is a smooth change of variables, so is $\Psi$.
We now proceed to show that $\bPsi$ sends solutions of~\eqref{eq:autoper} to solutions of~\eqref{eq:averaged}.

Consider a solution $ \big(\vect u(t),\thera(t)\big)$ of the differential equation 
\begin{equation}
	\begin{split}
		{\vect u}'(t) &= \matA\big(\tau(t)\big)\vect u(t) \\
		{\tau}'(t) &= 1
	\end{split}
\end{equation}
Define $t_0 = -\thera(0)$.
Clearly we then have $\thera(t) = t - t_0$ and $\vect{u}'(t) = \matA(t-t_0)\vect{u}(t)$.
By defining $\vect w(\sigma) \coloneqq \vect u(\sigma + t_0)$ we obtain
\(
\vect w'(\sigma) =  {\vect u}'(\sigma + t_0) =    \matA(\sigma)\vect{u}(\sigma + t_0) = \matA(\sigma)\vect{w}(\sigma)
\),
so $\vect w$ is a solution of~\eqref{eq:nonautonomous}.
As a result, $\vect w(\sigma) = \flow(\sigma)\vect w(0)$, which, using $\vect u(t) = \vect w\big(\thera(t)\big)$, gives 
\(
\vect u(t) =  \flow\big(\thera(t)\big) \vect u(t_0)
\).
We thus obtain that along a solution $\big(\vect u(t),\thera(t)\big)$,
\begin{equation}
\label{eq:phiinvuz}
\flow\big(\thera(t)\big)\inv \vect u(t) = \vect u(t_0)
.
\end{equation}
As a result, we obtain
\begin{equation}
	{\vect v}'(t) = \matAbar \exp(\matAbar\thera) \vect u(t_0) = \matAbar \vect v(t)
,
\end{equation}
which proves the result for the mapping $\bPsi$ and thus also for $\Psi$.
\end{proof}

\subsection{Reversibility} 
\label{sub:step_3}

We now turn to the question of whether the mapping $\Psi$ defined in \autoref{th:integrability} can preserve a reversibility structure.
We namely equip the space $\RR^n\times\Torus$ with a linear involution $R$, defined as
\begin{equation}
\label{eq:Rdef}
R(\vect u,\theta) \coloneqq (\rho \vect u, -\theta)
,
\end{equation}
for a given linear involution $\rho$.

We first observe the following condition on $\rho$ which will turn out to be essential for the preservation of the reversibility structure $R$ in \autoref{prop:reversible}.
\begin{lemma}
The $\g$-system (\autoref{def:autoper}) is reversible with respect to~$R$ if and only if
\begin{equation}
\label{eq:Areversible}
\rho \matA(\theta) \rho = -\matA(-\theta)
.
\end{equation}
\end{lemma}
\begin{proof}
  Let $f(\vect u, \theta) \coloneqq (\matA(\theta) \vect u, 1)$ be the vector field defining the differential equation \eqref{eq:autoper}.
  As $R$ is linear, reversibility with respect to $R$ means that $-f(\vect u, \theta) = R f\paren[\big]{R(\vect u, \theta)}$.
  Note that $f(R(\vect u,\theta)) = f(\rho \vect u, -\theta) = (\matA(-\theta)\rho\vect u,1)$.
  This gives the condition $(-\matA(\theta) \vect u, -1) = \paren{\rho\matA(-\theta)\rho \vect u, -1}$, which proves the claim.
\end{proof}

We now turn to the main result of this section.

\begin{theorem}
\label{prop:reversible}
Assume that~\eqref{eq:Areversible} holds.
Assume further that  the average matrix $\matAbar \in\g$  defined in \autoref{th:integrability} is such that
\begin{equation}
\label{eq:nohalfrotation}
\matAbar \subset \Omega
,
\end{equation}
where $\Omega$ is a subset of $\g$ where the exponential is injective, and such that $-\Omega \subset \Omega$, and that $\rho \Omega \rho \subset \Omega$.
Then the averaging mapping $\Psi\colon \RR^{n}\times\Torus\to\RR^{n}\times\Torus$, defined in \autoref{th:integrability}, preserves reversibility, i.e., $R\circ\Psi = \Psi\circ R$.
\end{theorem}


\begin{proof}

Recall that the flow map $\flow$ is defined by~\eqref{eq:defflowmap}.
Using~\eqref{eq:Areversible}, one shows that
\begin{equation}
\rho\flow(-\tau) =  \flow(\tau) \rho
.
\end{equation}
\itodo*[OV]{Possibly prove the claim above}
This implies in particular that
\(
\rho \flow(-1) = \perPhi \rho
\),
so using~\eqref{eq:Floquet}, we obtain
\begin{equation}
\label{eq:mirrorrot}
\rho \perPhi\inv = \perPhi \rho
.
\end{equation}
Now, recalling the definition of $\matAbar$ in~\eqref{eq:defmatAbar}, we notice that~\eqref{eq:mirrorrot} implies that
\begin{equation}
\exp(-\matAbar) = \rho\exp(\matAbar)\rho = \exp(\rho\matAbar\rho)
.
\end{equation}
As $-\matAbar \in -\Omega \subset \Omega$ and $\rho\matAbar\rho \in\rho\Omega\rho \subset \Omega$, we use the injectivity of the exponential and deduce that $-\matAbar = \rho \matAbar \rho$.
We therefore obtain
\begin{equation}
\label{eq:commutexprho}
\exp(-\matAbar \thera) \rho = \rho \exp(\matAbar\thera)
.
\end{equation}
By combining~\eqref{eq:mirrorrot} and~\eqref{eq:commutexprho} we get
\begin{equation}
\begin{split}
\vect v(\rho \vect u, -\thera) &= \exp(-\matAbar\thera)\flow(-\thera)\inv \rho \vect u  \\
&= \exp(-\matAbar\thera) \rho \flow(\thera)\inv \vect u \\
&= \rho \exp(\matAbar \thera) \flow(\thera)\inv \vect u
\end{split}
\end{equation}
so we finally obtain
\begin{equation}
\vect v(\rho \vect u, -\thera) = \rho \vect v
.
\end{equation}
This finishes the proof.
\end{proof}

\section{Fibrations over reversible integrable systems}\label{sec:fibrations_over_int_sys}

The goal of this section is to show that all the nonholonomically coupled systems studied in \autoref{sec:nonholexamples} are fibrated over a reversible integrable system, as summarized in \autoref{tab:summary}.
The situation is illustrated in \autoref{fig:spiral}.

\subsection{Reversible integrability}

\begin{figure}
	\centering
%
%

\newcommand*\spircircle[3]{
{	
\newcommand*\prad{#1}
\newcommand*\zvel{#2}

\addplot3+[,ytick=\empty,yticklabel=\empty,
  mark=none,
  thick,
  #3,
  domain=10:14.7*pi,
  samples=400,
  samples y=0,
]
({\prad*sin(0.28*pi*deg(x))},{\prad*cos(0.28*pi*deg(x)},{\zvel*x});

\addplot3+[,ytick=\empty,yticklabel=\empty,
  mark=none,
  thick,
  #3,
  domain=0:1,
  samples=400,
  samples y=0,
]
({\prad*sin(2*pi*deg(x))},{\prad*cos(2*pi*deg(x)},{0});
}
}

\begin{tikzpicture}
\begin{axis}[
 view={-10}{-30},
 axis lines=middle,
 zmax=60,
 hide axis,
 xmin = -15,
 xmax=20,
 ymin = -15,
 ymax=15,
 height=6cm,
 xtick=\empty,
 ytick=\empty,
 ztick=\empty
]

\spircircle{3}{1.7}{Black}
\spircircle{8}{2}{Gray}

\newcommand*\extent{12}
\draw (axis cs:\extent,-\extent,0) -- (axis cs:\extent,\extent,0) -- (axis cs:-\extent,\extent,0) -- (axis cs:-\extent,-\extent,0) -- cycle;

\node at (axis cs:\extent,-\extent,0) [right] {$\Downstairs$};



%
\end{axis}
\end{tikzpicture}

	\caption[]{
		An illustration of the setting of the examples treated in this paper.
		The actual system is represented by the spirals in the manifold $\Upstairs$.
		The system, however, sits above an integrable system, which foliation in tori is represented downstairs on $\Downstairs$.
		If the numerical integrator descends to a reversible integrator downstairs, then there is no drift in the first integral of that system.
		Moreover, as the energy depends only on these invariants, there is no energy drift either.
	}
\label{fig:spiral}
\end{figure}
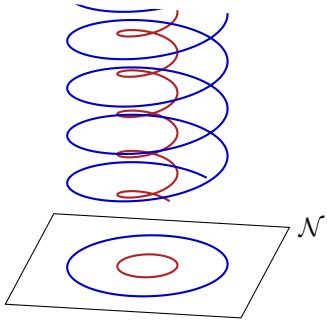

\newcommand*\action{I}

We briefly recall the definition of an \emph{$R$-integrable}, or \emph{reversible integrable}, system~\cite{Se1998}.

\begin{definition}
\label{def:Rintegrable}
Consider a manifold $\Downstairs$,  equipped with an involution $R$.
A dynamical system, i.e., a vector field $Y$ on the manifold $\Downstairs$, is \emph{$R$-integrable} if there is a map $\varphi \colon \Downstairs \to \RR^p \times (\RR/\ZZ)^k $, which we denote by $\varphi(x) = (\action(x), \theta(x))$, such that
\begin{enumerate}[label=\upshape(\roman*)]
	\item  $\varphi \circ R = (\action, -\theta)$
	\item  the induced vector field is $\dot\action = 0$ and $\dot\theta = \omega(\action)$ for a given \emph{frequency map} $\omega$
  \item
    \label{def:nondegenerate}
    the image of the frequency map $\omega$ does not lie in a proper linear subspace.
\end{enumerate}
\end{definition}

\begin{remark}
  The last condition~\ref{def:nondegenerate} is called a \emph{nondegeneracy}, or \emph{non resonance}, condition~\cite{Se1998}.
  Note that the usual non resonance condition (also called \emph{diophantine} condition~\cite[\S X.2.1]{HaLuWa2006}) is not strong enough for our examples, as discussed in~\cite{MoVe2014}.
\end{remark}

We first make a statement about some $\g$-periodic systems (\autoref{def:autoper}).

\begin{proposition}
\label{prop:integsys}
	Consider a $\g$-periodic system.
	We assume that the average matrix $\matAbar$ from \autoref{th:integrability} is either zero, or, if $n \leq 3$, an element of $\so(n)$.
	Suppose further that there exists a map $\rho$ fulfilling~\eqref{eq:Areversible}.
	Then the system is $R$-integrable (\autoref{def:Rintegrable}), with $R$ defined in~\eqref{eq:Rdef}.
\end{proposition}
\begin{proof}
	The assumptions of \autoref{th:integrability} and \autoref{prop:reversible} are fulfilled, so we obtain a variable transformation which brings the system to the form~\eqref{eq:averaged}.
	If $\matAbar$ is zero, there are only action variables.
	If $\matAbar$ is in $\so(n)$ for $n \leq 3$, we have one more angle variable determined by the only angle of the rotation matrix $\exp(\matAbar)$.
\end{proof}

\begin{proposition}
\label{prop:redrevinteg}
All the reduced systems defined in \autoref{sec:nonholexamples} fulfill the assumptions of \autoref{prop:integsys}.
\end{proposition}
\begin{proof}
  The only nontrivial case is that of the knife edge, where the group is $\RR$.
The average matrix $\matAbar$ computed in \autoref{prop:reversible} can be computed explicitly in that case:
\begin{align}
  \matAbar =
\begin{bmatrix}
	0&\overline{\kker}\trans{\FF}\\
	0&0
\end{bmatrix}
\end{align}
where
\begin{align}
	\overline{\kker} \coloneqq \frac{1}{2\pi}\int_0^1 \kker(\zz)\ud \zz
  ,
\end{align}
which, following \eqref{eq:kinfeedgekernel}, is simply zero.

We define the reversibility structure $R$ from \eqref{eq:Rdef}, where $\rho$ is defined as follows.
For the CVT, $\rho$ is
\begin{equation}
	\rho(\yy,\vv) = (\yy,-\vv)
  .
\end{equation}
For the remaining systems, $\rho$ is
\begin{equation}
	\rho(\vv) = \vv
  .
\end{equation}
In both cases, $\rho$ fulfils \eqref{eq:Areversible}.


\end{proof}









\section{Conclusions and open problems}\label{sec:outlook}

In this section we summarize the main points of our paper and, based on our findings, list a set of open problems.
These problems are aimed for the nonholonomic integrators community; our goal is that work on them shall lead to a better understanding of geometric integration algorithms for nonholonomic systems.

We start with the conclusions.
The field of nonholonomic integrators aims to construct numerical methods specifically designed for nonholonomic systems, aspiring to outperform standard integration schemes. 
The starting point for designing such integrators is most often some discrete emulation of the Lagrange-d'Alembert principle.
Contrary, however, to the holonomic case, where Hamilton's principle leads to conservation of symplecticity, the phase space flow structure induced by the Lagrange-d'Alembert principle is not fully understood; total energy is always conserved but are there any additional structures?
This lack of understanding is reflected in the literature on nonholonomic integrators.
Indeed, several different, incompatible discrete versions of the Lagrange-d'Alembert principles have been suggested and it is unclear if one or the other yield better performance. 
(Compare, for example, the DLA-principle in \cite{CoMa2001, McPe2006} with the GNI approach in \cite{FeIgMa2008}.)
All the suggested discrete frameworks do, however, have a common feature: if the nonholonomic system is reversible (with respect to the standard reversibility map), then the discrete phase flow map preserves reversibility.
Since all standard nonholonomic test problems are reversible (and typically also integrable), one can expect that the observed numerical behaviour is explained by theory for reversible dynamical systems (for example reversible KAM theory).
In this paper we have concluded this to be the case for a large class of nonholonomic test problems (fibrations over reversible integrable systems).
To verify that reversibility alone is responsible for the good long-time behaviour we have developed a new class of perturbed nonholonomic test problems that are still integrable, but no longer reversible with respect to the standard reversibility map.
Our simulations show that none of the tested nonholonomic integrators perform well on all of these problems.
Of course, our specific selection of test problems does not cover all nonholonomic problems studied in the literature.
They do, however, in many ways represent the simplest non-trivial nonholonomic systems and are therefore worthy for primary investigations of nonholonomic integrators (the notion being that before moving on to more complicated problems, an integrator should perform well on the simplest non-trivial problems).

In addition to nonholonomic integrators based on discrete emulation of the Lagrange-d'Alembert principle, another approach is to start directly from conservation of energy and momentum and construct algorithms that exactly preserve these conservation laws (\emph{energy-momentum methods}, see e.g.\ \cite{Be2006}).
Our numerical simulations with a commonly used method from this class shows that, although both energy and reversibility are preserved, the fibration structure of the test problems is not preserved, even for the reversible problems, which leads to a drift in the ``hidden'' integrals.
This numerical observation is in agreement with our theory based on reversible KAM theory, which cannot be applied unless the fibration structure is also preserved.
We thus conclude that forcing conservation of first integrals is not enough to obtain good long-time behaviour.

\subsection*{Open problems} 
We have formulated a set of problems that reflect, in our opinion, the core challenges for the nonholonomic integrator community.
The problems are listed from most to least general.

\begin{enumerate}

	\item \textbf{Backward error analysis for relevant classes of nonholonomic systems.}
	In lack of a complete characterization of nonholonomic systems, one can still develop a numerical \emph{backward error analysis} (cf.~\cite{HaLuWa2006}) for subclasses of nonholonomic dynamics where the (reduced) phase space geometry is understood. 
	We suggest to start with the class of integrable nonholonomically coupled systems described in this paper.
	Thus, the problem consists in developing nonholonomic integrators such that, when they are applied to integrable nonholonomically coupled systems, their modified vector fields on the reduced phase space preserve the structure of being fibrated over integrable systems.

	\item \textbf{Explain leap-frog performance in \autoref{fig:energy_cvt}.} 
	The lower right diagram of \autoref{fig:energy_cvt} indicates that the leap-frog method nearly conserves the driver energy of a rotating, non-reversible driver.
	As pointed out, this is the only near conservation behaviour we have seen that is not (directly) explained by reversible KAM theory.
	The open problem is to come up with a rigorous understanding of why near conservation is observed in this case.
	A hypothesis is that reversible KAM still can be used, but that one needs to slightly modify the reversibility map which happen to be conserved by leap-frog (but not the other integrators).
	A rigorous understanding of the unexplained behaviour is likely to shed light on the structure of nonholonomic systems.
	(We remark again that leap-frog yields drift in integrals for other non-reversible integrable systems, for example the perturbed knife edge as seen in \autoref{fig:energy_knife_edge}.)

	\item \textbf{Energy-momentum methods that preserve fibrations} 
	As seen in 	\autoref{fig:latitude_cvt}, the energy preserving method (DD), despite being reversible, performs poorly on the reversible CVT problem (it exhibits drift in the ``hidden'' first integral).
	The reason for the poor performance has to do with this integrator not perserving the fibration structure.
	The last open problem consists in finding an energy-momentum nonholonomic integrator that also preserves the fibration structure.
	To base the integrator on the average vector field approach, instead of discrete derivatives, might lead to a solution.
\end{enumerate}

\appendix

\section{Integration schemes for nonholonomic systems}
\label{sec:integrators}

\settocdepth{section}

In this section we give a brief description of the numerical integrators specifically developed for nonholonomic systems of the form~\eqref{eq:eq_of_motion}.
With \emph{integrator} we mean a map 
\begin{equation}
	\Phi_h\colon (\qq_{k},\dot\qq_{k})\mapsto (\qq_{k+1},\dot\qq_{k+1}),
\end{equation}
depending on the time-step parameter $h>0$, such that the sequence 
\begin{equation}
	\qq_1,\,\qq_2,\,\ldots	
\end{equation}
with $\qq_0 = \qq(0)$, approximates the exact solution sequence 
\begin{equation}
	\qq(h),\,\qq(2h),\,\ldots	
\end{equation}
for the nonholonomic system under consideration.

An integrator for some nonholonomic system is called \emph{reversible} if 
\begin{equation}\label{eq:reversible_integrator}
	\Phi_h^{-1}(\qq,\dot\qq) = \Phi_h(\qq,-\dot\qq).	
\end{equation}
Let us give a word of caution here; the notion of reversibility of an integrator depends on the particular non-holonomic system being approximated.
Indeed, a method may be reversible for one non-holonomic system, but fail to be so for another system.

\subsection{Discrete Lagrange--d'Alembert (DLA) integrators}
\label{sub:dla}

The \emph{discrete Lagrange--d'Alembert (DLA) principle} for deriving nonholonomic integrators was introduced by~\citet{CoMa2001}.
The notion is based on a nonholonomic version of \emph{discrete Lagrangian mechanics}, developed by~\citet{Ve1991}.
By specifying a discrete Lagrangian function and a discrete constraint manifold, the DLA principle yields an integration algorithm, 
typically by an implicit equation (see~\cite[\S\!~3]{CoMa2001} for details).

\citet{CoMa2001} suggested the \emph{\DLA{\alpha} integrator}:
\begin{equation}\label{eq:dla_alpha}
	\begin{split}
		\qq_{1-\alpha} &= \qq_{0} + (1-\alpha) h \dq_{0} \\
		\dq_{1} &= \dq_{0} - \alpha h\nabla V(\qq_{0}) - (1-\alpha) h\nabla V(\qq_{1}) + hA(\qq_{1-\alpha})^T\blambda \\
		\qq_{1} &= \qq_{1-\alpha} + \alpha h \dq_{1} \\
		0 &= A(\qq_{1})\dq_{1}.
	\end{split}
\end{equation}
This method is first-order accurate for $\alpha \neq 1/2$ and second-order accurate for $\alpha = 1/2$.
Explicitly, \DLA{\frac{1}{2}} is given by
\begin{equation}\label{eq:dla_half}
	\begin{split}
		\qq_{1/2} &= \qq_{0} + \frac{h}{2}\dq_{0} \\
		\dq_{1} &= \dq_{0} - \frac{h}{2}\nabla V(\qq_{0}) - \frac{h}{2}\nabla V(\qq_{1}) + hA(\qq_{1/2})^T\blambda \\
		\qq_{1} &= \qq_{1/2} + \frac{h}{2}\dq_{1} \\
		0 &= A(\qq_{1})\dq_{1}.
	\end{split}
\end{equation}

Another DLA method, suggested by \citet{McPe2006}, is obtained by taking half a step with \DLA{0} followed by half a step with \DLA{1}.
This \emph{\DLA{0,1} integrator} is given by
\begin{equation}\label{eq:dla_01}
	\begin{split}
		\qq_{1/2} &= \qq_{0} + \frac{h}{2}\dq_{0} \\
		\dq_{1} &= \dq_{0} - h\nabla V(\qq_{1/2}) + hA(\qq_{1/2})^T\blambda \\
		\qq_{1} &= \qq_{1/2} + \frac{h}{2}\dq_{1} \\
		0 &= A(\qq_{1})\dq_{1}.
	\end{split}
\end{equation}
This method is second-order accurate and, contrary to \DLA{\frac{1}{2}}, only linearly implicit.

\subsection{Nonholonomic leap-frog method}
\label{sub:leapfrog}

The nonholonomic leap-frog method is given by
\begin{align}
  \qq_{1} &= \qq_0 + h \dq_{1}\\
  \dq_{1} &= \dq_0 + h (\nabla V(\qq_0) + \consmat(\qq_0)^{T} \blambda)\\
  0 &= \consmat(\qq_0)(\dq_0+\dq_{1})
\end{align}
This integrator was suggested already in~\cite[\S\,5.1]{McPe2006}, but discarded as it is not based on the DLA principle.
A very important property of this method is that it is only \emph{linearly-implicit}.

As observed in~\cite{FeIgMa2008}, this method fits within the so-called geometric nonholonomic integrator (GNI) family~\cite{FeIgMa2008,FeIgMa2009,KoMaDiFe2010,FeJiMa2015}, when it is rewritten in the following way:
\begin{equation}\label{eq:leap_frog}
	\begin{split}
		\dq_{1/2} &= \dq_{0} - \frac{h}{2}\left( \nabla V(\qq_0) + A(\qq_{1/2})^T\blambda_{0} \right) \\
		\qq_{1} &= \qq_0 + h \dq_{1/2} \\
		\dq_{1} &= \dq_{1/2} - \frac{h}{2}\left( \nabla V(\qq_{1}) + A(\qq_{1})^T\blambda_{1} \right) \\
		0 &= A(\qq_{1})\dq_{1}.
	\end{split}
\end{equation}


\subsection{Discrete derivative nonholonomic integrators}
\label{sec:epinteg}

\newcommand*\dnabla{\overline{\nabla}}

These integrators are energy preserving, and are based on the use of discrete gradients~\cite{Be2006}.

Let us define the auxiliary variable
\begin{equation}
 \overline{\qq} \coloneqq \frac{\qq_0 + \qq_1}{2} 
 .
\end{equation}
We assume that a function $\dnabla V(\qq_0, \qq_1)$
is a \emph{discrete gradient} of the potential, 
that is, it fulfills the requirement $\dnabla V(\qq_0,\qq_1) (\qq_1 - \qq_0) = V(\qq_1) - V(\qq_0)$.
The method is then described by:
\begin{align}
  \qq_1 &= \qq_0 + h \frac{\dot\qq_0 + \dot\qq_1}{2} \\
  \dot{\qq}_1 &= \dot{\qq}_0 + h \paren[\Big]{-\dnabla V \paren*{\qq_0, \qq_1}  + \consmat(\overline{\qq})^T \blambda} \\
  0 &= \consmat\paren*{\overline{\qq}}\frac{\dot\qq_0 + \dot\qq_1}{2}
\end{align}
It is straightforward to check that the energy is exactly preserved by this method.






\bibliographystyle{amsplainnat} 
\bibliography{elbar2}

\end{document}